\documentclass{amsart}
\usepackage{cite}

\newtheorem{theorem}{Theorem}[section]
\newtheorem{lemma}[theorem]{Lemma}
\newtheorem{corollary}[theorem]{Corollary}

\theoremstyle{definition}

\theoremstyle{remark}
\newtheorem{remark}[theorem]{Remark}

\numberwithin{equation}{section}

\include{amsmath}

\newcommand{\p}{\mathbf{p}}
\renewcommand{\Re} {\operatorname{Re}}

\begin{document}

\title[$p$-adic  fractal string of rational dimensions and Cantor strings]
 {$p$-adic fractal strings of arbitrary rational dimensions and Cantor strings}

\author{Michel L. Lapidus}
\address{University of California, Riverside\\
Department of Mathematics\\
900 Big Springs Road\\
Riverside, CA 92521-0135 USA}
\email{lapidus@math.ucr.edu}
\thanks{The work of the first author (MLL) was partially supported by the US National Science Foundation (NSF) under the research grants DMS-0707524 and DMS-1107750, and by the Burton Jones Endowed Chair in Pure Mathematics (of which MLL was the chair holder at the University of California, Riverside, during the completion of this paper). }

\author{L\~u' H\`ung}
\address{C\^ongfu Garden}
\email{satnamo@youtube}

\author{Machiel van Frankenhuijsen}
\address{Utah Valley University\\
Department of Mathematics\\
800 West University Parkway\\
Orem, Utah 84058-5999 USA}
\email{vanframa@uvu.edu}

\subjclass[2000]{Fractal Geometry, Number Theory, $p$-adic Analysis}

\date{\today}

\dedicatory{This paper is dedicated to our advisors with gratitude and admiration.}
\keywords{Fractal geometry, number theory, $p$-adic analysis, fractal strings, zeta functions, complex dimensions, ad\`elic Cantor, Cantor--Smith, Euler and Euler--Riemann strings.}

\begin{abstract}
The local theory of complex dimensions for real and $p$-adic fractal strings describes oscillations that are intrinsic to the geometry, dynamics and spectrum 
 of archimedean and nonarchimedean fractal strings. 
We aim to develop a global theory of complex  dimensions for ad\`elic fractal strings in order to reveal the oscillatory nature of  ad\`elic fractal strings 
and to understand the Riemann hypothesis in terms of the vibrations and resonances of fractal strings. 

We present a simple and natural construction of self-similar $p$-adic fractal strings 
of any rational dimension in the closed unit interval $[0,1]$.
Moreover, as a first step towards a global theory of complex dimensions for ad\`elic fractal strings, 
 we construct an ad\`elic Cantor string  in the set of finite ad\`eles $\mathbb{A}_0$
as an infinite Cartesian product of every $p$-adic Cantor string,
as well as an ad\`elic Cantor--Smith string in the ring of ad\`eles $\mathbb{A}$
as a Cartesian product of the general Cantor string and the ad\`elic Cantor string. 
\end{abstract}

\maketitle

\tableofcontents

\section{Introduction}

The theory of complex  dimensions is a theory of oscillations that are intrinsic to fractal geometries. 
Complex dimensions provide a natural tool to reveal the oscillatory nature of fractality. 
Geometrically, a fractal is like a musical instrument tuned to play certain notes with frequencies essentially equal to the imaginary parts of the underlying complex dimensions
and with amplitudes essentially equal to the real parts of the underlying complex  dimensions.  
Physically, we can imagine a geometric wave propagating through the `space of scales' that lies beneath the shimmering surface of a fractal 
with the aforementioned frequencies and amplitudes.
Fractality is therefore defined as the presence of nonreal complex dimensions in the theory of complex dimensions for fractal strings and compact subsets of Euclidean spaces \cite{Lap4, Lap5, LvF1, LvF2, LvF3, LRZ}.

Following the examples of the $a$-string and of the ordinary Cantor string given by the first author in the early 1990s \cite{Lap1, Lap2, Lap3, LapPo1}, 
the notion of a fractal string was conceived and defined by the first author and Carl Pomerance in their investigation and resolution 
of the one-dimensional Weyl--Berry conjecture for fractal drums and its connection with the Riemann zeta function in 1993,
 \cite{LapPo1}.
The Riemann hypothesis for the Riemann zeta function turned out to be equivalent to the solvability of the inverse spectral problems for fractal strings, 
as was established by the first author and Helmut Maier in 1995,
 \cite{LapMa1}. 
The idea of complex dimensions started to emerge out of the work in \cite{Lap1, Lap2, Lap3, LapPo1} 
and was used heuristically in \cite{LapMa1} in the authors' 
spectral reformulation of the Riemann hypothesis for the Riemann zeta function~$\zeta(s)$
by means of  a pair of complex conjugate numbers $\omega=\sigma+it$ and~$\bar{\omega}=\sigma -it$
 that lie symmetrically above and below the real fractal dimension~$D=\sigma$ of a fractal string~$\mathcal L$,
where the geometric oscillations of order $\sigma$ 
disappear in the spectrum of~$\mathcal L$
if~$\zeta(\omega)=0$. 
The precise notion of complex dimensions $\omega \in \mathbb C$, 
defined as the visible poles $\omega $ of the geometric zeta function $\zeta_{\mathcal{L}}(s)$ associated with a real fractal string~$\mathcal{L}$,
was crystallized and rigorously developed by the first and third authors in \cite{LvF1} and then significantly extended in \cite{LvF2, LvF3},
as well as in many other works. 
 
A higher-dimensional theory of complex dimensions for compact subsets of Euclidean spaces was initiated by the first author and Erin Pearse
 in a direct calculation of a tube formula for the Koch snowflake curve in 2006,
 \cite{LapPe1},
then pursued by the same authors and Steffen Winter in 2011,
 \cite{LPW},
 and finally fully developed
 in the~2017 research monograph \cite{LRZ}
 by the first author,
Goran Radunovi\'c, and Darko $\check{Z}$ubrini\'c 
 after the introduction of the distance zeta function by the first author in 2009 
 that was inspired, in part, by a result of Reese Harvey and John Polking in their study of the removable singularities of solutions of certain linear partial differential equations \cite{HP}.

Motivated by the 
  $p$-adic string theory in $p$-adic mathematical physics \cite{VVZ} 
and the vision of a global theory of complex dimensions for ad\`elic fractal strings, 
the geometric theory of complex dimensions for fractal strings was extended to the nonarchimedean field of $p$-adic numbers $\mathbb{Q}_p$ by the authors of the articles 
\cite{LapLu1, LapLu2, LapLu3, LLvF1, LLvF2}. 
Similar to the theory of complex dimensions for real fractal strings, 
the theory of complex dimensions for $p$-adic fractal strings reveals the oscillatory nature intrinsic to the geometry of $p$-adic fractal strings  
via the explicit tube formulas for the volume of the inner $\varepsilon$-neighborhoods of $p$-adic fractal strings \cite{LapLu3, LLvF1, LLvF2}. 

A basic example of a fractal is the Cantor set, 
which is nothing with respect to the one-dimensional Lebesgue measure on the real line
and everything with respect to the zero-dimensional measure that counts points in a set since it contains as many points as the whole real line, 
yet itself contains no interval of nonzero length. 
The simplest construction of the Cantor set is by successively removing all the open middle thirds of the closed unit interval $[0, 1]$.
The Cantor set is an uncountably infinite set of points in $[0,1]$ that has zero length. 
Irrational numbers have the same property, but the Cantor set is closed, 
so it is not even dense in any interval, 
whereas the irrational numbers are dense in every interval. 

In order to illuminate the difficulties of integrating discontinuous functions,
Henry John Stephen Smith constructed a general Cantor set in 1875 by 
successively removing all $m-1$ segments from $m>2$ equal parts of the interval $[0,1]$;
``and exempt the last [open] segment from any subsequent division'' in every iteration \cite{Smith}.
Georg Ferdinand Ludwig Philipp Cantor independently 
introduced the Cantor ternary set in 1883
as the set of real numbers of the form
\[x=\frac{c_1}{3}+\cdots+\frac{c_v}{3^v}+\cdots,\]
where the digit $c_v$ is 0 or 2 for each integer $v\geq 1$,
as an example of a perfect set that is not everywhere dense;
see~\cite{Cantor,FJ}.
Through careful considerations of the Cantor set,
 Georg Cantor created set theory and helped lay the foundation of point-set topology. 
\begin{quote}
Aus dem Paradies, das Cantor uns geschaffen hat, soll uns niemand vertreiben k\"onnen \cite{AF}.\footnote{From his paradise Cantor with us unfolded,
 we hold our breath in awe,
 knowing we shall not be expelled.}
David Hilbert (1862--1943).
\end{quote}

In closing the introduction, we give an overview of the remaining of this article.
In the preliminary~\S2,
we describe the finite and infinite valuations on the field of rational numbers~$\mathbb Q$ and their correspondence with the absolute values on $\mathbb Q$. 
We also describe the nonarchimedean field of $p$-adic numbers $\mathbb{Q}_p$ 
and the ring of ad\`eles $\mathbb{A}$. 
In \S3, for each prime number $p$, 
we construct a self-similar $p$-adic fractal string $\mathcal{L}_p$ of dimension~$D$
and with oscillatory period $\p=\frac{2\pi}{m\log p}$,
 for any rational number $D=\frac{k}{m}\in [0,1]$.
The construction of the $p$-adic fractal string of dimension $D=1/2$ is especially interesting since it involves the diagonal of the digits.
In~\S4,
 for each prime number $p>2$,
we construct a $p$-adic Cantor string of dimension $D=\frac{\log(\frac{1+p}{2})}{\log p}$
and with oscillatory period $\p=\frac{2 \pi}{\log p}$. 
We also construct an ad\`elic Cantor--Smith string in $\mathbb{A}$ 
as an infinite product of the general Cantor string and every $p$-adic Cantor string. 
In \S5,
we construct an ad\`elic Euler--Riemann string $\mathcal{E}_{\mathbb A}$ in $\mathbb A$
with the Riemann zeta function as a possible geometric zeta function of $\mathcal{E}_{\mathbb A}$.
In \S6,
we envision that 
a natural global theory of complex dimensions for ad\`elic fractal strings would
provide a unified framework for understanding the vibrations and resonances in the geometry and the spectrum of ad\`elic fractal strings.

\section{Preliminaries}

By Ostrowski's theorem, 
the global field of rational numbers $\mathbb Q$ has,
up to equivalence,  
one archimedean valuation $v_{\infty}(x)=\log|x|_{\infty}$ and infinitely many nonarchimedean valuations $v_{p}(x)= -\mbox{ord}_p(x)\log p$, 
one for each prime number $p \in \mathcal P$,
 where $\mbox{ord}_p(x)$ is the power of $p$ in the prime factorization of a nonzero rational number 
 $x\in \mathbb{Q}^*$.\footnote{Here, $|x|_{\infty}$ denotes the usual absolute value of $x\in \mathbb{Q}^*$.}
The valuation of a number $x$ is positive when $x$ is far away from~$0$ in the $v$-adic topology
or more concretely,
 if the denominator of $x$ is divisible by~$p$.
For every number $x\in \mathbb Q^*,$
 we have the important \emph{Artin--Whaples sum formula\/}:
\[
\sum_{p\leq \infty}v_p(x)=0.
\]

There is a one-to-one correspondence between valuations and (equivalent classes of) absolute values on the field of rational numbers $\mathbb Q$.\footnote{Two absolute values on $\mathbb Q$ are said to be equivalent if one is a positive power of the other.}
Each valuation $v$ on $\mathbb Q$
 is assigned to a unique absolute value $|\cdot|_v$ by defining
$|x|_v=e^{v(x)}$,
 for any nonzero number $x\in \mathbb Q^*$,
 and $|0|_v=0$.
Conversely, every absolute value $|\cdot|$ on $\mathbb Q$ is assigned to a unique valuation~$v_{|\cdot|}$ by defining 
$v_{|\cdot|}(x)=\log |x|$,
 for all $x\in \mathbb Q$. 
Hence, the \emph{Artin--Whaples sum formula} becomes the \emph{Artin--Whaples product formula} via this correspondence:
\[
\prod_{p\leq \infty}|x|_p=1.
\]

A physical version of the id\`elic product formula was discovered by Peter Freund and Edward Witten in 1987 \cite{FW}:
the ordinary Veneziano string amplitude for the scattering of four open bosonic strings in their tachyon states is the inverse product of all its $p$-adic Veneziano string amplitudes;
\[
\prod_{p\in\mathcal P}\int_{\mathbb{Q}_p}|x|_p^a\cdot|1-x|_p^b~dx=1.
\]
``This factorization is equivalent to the functional equation for the Riemann zeta function'' \cite{FW}.

The topological completion of $\mathbb Q$ with respect to the metric topology induced by 
the infinite place of $\mathbb Q$
 is the local archimedean field of real numbers $\mathbb R=\mathbb{Q}_{\infty}$.
Similarly, the topological completion of $\mathbb Q$ with respect to the ultrametric topology induced by a finite place
is the local nonarchimedean field of $p$-adic numbers $\mathbb{Q}_p$.

The (closed) unit ball in the nonarchimedean field of $p$-adic numbers 
$\mathbb{Q}_p$ is the \emph{ring} of $p$-adic integers $\mathbb{Z}_p$;
it is closed under multiplication \emph{and addition}.
 Moreover, 
every $p$-adic number $x\in \mathbb{Q}_p$ has a unique $p$-adic expansion as
\[
x=a_kp^k+\cdots + a_0 + a_1p + a_2p^2+\cdots,\]
for some integer $k\in \mathbb Z$ such that $a_k\neq 0$ and where all the digits $a_j\in\{0, 1, \ldots, p-1\}.$
In particular, every $p$-adic integer $y\in \mathbb{Z}_p$ has a unique $p$-adic expansion as
\[y=a_0 + a_1p + a_2p^2+\cdots.\]

The set of finite ad\`eles $\mathbb{A}_0$ is the set of sequences of $p$-adic numbers $x_p$ indexed by the  prime numbers such that $x_p\in \mathbb{Z}_p$,
 for almost all primes (i.e.,
for all but finitely many primes):
\[
\mathbb{A}_0=\{ (x_2, x_3, x_5, \ldots) ~|~ x_p\in \mathbb{Q}_p \mbox{ for all $p$ and $x_p\in \mathbb{Z}_p$ for almost all } p\}.
\]
The set of ad\`eles $\mathbb A$ also has an archimedean component, 
\[
\mathbb A=\mathbb R \times \mathbb{A}_0.
\]

Topologically,
 the set of finite ad\`eles $\mathbb{A}_0$ is zero-dimensional because it is totally disconnected, but combinatorially,
 it is
 one-dimensional, just like $\mathbb{Z}_p$;
see Remark~\ref{RM1}.
On the other hand, the set of ad\`eles $\mathbb A$ is one--dimensional, both topologically and combinatorially. 
Therefore, every ad\`elic fractal string based on either $\mathbb{A}_0$ or $\mathbb A$ has fractal dimension at most 1. 
This Minkowski--Bouligand dimension gives combinatorial information about the number of balls required to cover the set. 
Here, balls are translates of the basic neighborhoods of zero, 
\[
\prod_{p\in \mathcal P}p^{r_p}\mathbb{Z}_p,
\]
where the integral exponent $r_p\neq 0$ for only finitely many prime numbers $p$. 
These are the usual basic open sets that one uses to define the Lebesgue integral and the Fourier transform in number theory. For example, see Tate's thesis \cite{Ta}.

\section{$p$-adic fractal strings with rational dimensions between $0$ and $1$}

In looking for a geometric way to create an ad\`elic fractal string and a global theory of complex dimensions, 
we discovered a simple and natural construction of $p$-adic fractal strings of any rational dimension in the critical interval $[0, 1]$. 
The simplest example is of dimension $D=\frac{1}{2},$ which is particularly interesting since it involves the diagonal of the digits.

Given a prime number $p$ in the infinite set of primes $\mathcal P=\{2, 3, 5, 7, 11, \ldots \}$, let us 
consider the ring of $p$-adic integers $\mathbb{Z}_p$, taken two digits at a time:
\[
\mathbb{Z}_p=\{ (a_0 + b_0p) + (a_1 + b_1p)p^2 + (a_2 + b_2p)p^4 + \cdots ~|~ a_j, b_j \in \{0, 1, 2, \ldots, p-1\}\}.
\]
Let $S$ be a subset of all such pairs of digits containing exactly $p$ pairs. 
For example, we can take $S$ to be the diagonal,
\[
S=\{ 0, 1+p, 2+2p, 3+3p, \ldots, p-1+(p-1)p\}.
\]
Then $\mathbb{Z}_p$ contains $p^2-p$ copies of $p^2\mathbb{Z}_p,$ namely, 
\[
a+bp+p^2\mathbb{Z}_p, \mbox{ for each }  a+bp\in \{0,1, 2, \ldots, p^2-1\}- S.
\]

These are the first $p^2-p$ substrings of a $p$-adic fractal string of length $p^{-2}$.
By continuing this process with the remaining $p$ subintervals $a+bp+p^2\mathbb{Z}_p$ for each $a+bp\in S$,
we obtain a self-similar $p$-adic fractal string $\mathcal{L}_p$ inside $\mathbb{Z}_p$ 
that is fixed by the iterated function system of similarity contraction mappings from $\mathbb{Z}_p$ into itself:
\[ f_{a+bp}(x)= a+bp+p^2x, \mbox{ for each } a+bp\notin S.\]
The lengths of the string $\mathcal{L}_p$ are 
$l_n=p^{-2n},$
 with multiplicity
 $\mu_n=(p^2-p)p^{n-1}$
for each positive integer $n$. 
Therefore, the geometric zeta function $\zeta_{\mathcal{L}_p}(s)$ is the meromorphic continuation to all of $\mathbb C$
of the following convergent geometric series of $\mathcal{L}_p$:
\[
\sum_{n=1}^{\infty} \mu_n\cdot l_n^s=
 \sum_{n=1}^{\infty}\frac{(p^2-p)p^{n-1}}{p^{2ns}}
=\frac{p^2-p}{p^{2s}-p},
\mbox{ for }\Re(s)>\frac{1}{2};
\]
that is to say,
\[
\zeta_{\mathcal{L}_p}(s)
=\frac{p-1}{p^{2(s-\frac{1}{2})}-1},
 \mbox{ for all } s\in \mathbb C. 
\]
Since the complex dimensions of a fractal string $\mathcal{L}_p$ are defined as the poles of the geometric zeta function $\zeta_{\mathcal{L}_p}(s)$, 
the set of complex dimension of $\mathcal{L}_p$ is
\[\mathcal{D}_{\mathcal{L}_p}=
\{D + in\p~|~ n\in \mathbb Z\},\]
where $D=\frac{1}{2}$ is the Minkowski--Bouligand dimension 
of  $\mathcal{L}_p$ 
and 
$\p=\frac{\pi}{\log p}$ is its oscillatory period \cite{LLvF2}.
The residue of  $\zeta_{\mathcal{L}_p}(s)$ at the midfractal dimension $D=\frac{1}{2}$ is
 \[\mbox{res}(\zeta_{\mathcal{L}_p}(s); \frac{1}{2})=\frac{p-1}{2\log p},\]
and this is also the residue of $\zeta_{\mathcal{L}_p}(s)$ at every other complex dimension $\omega \in \mathcal{D}_{\mathcal{L}_p}$,
by the periodicity of $\zeta_{\mathcal{L}_p}(s)$.

We thus have established the following result.  

\begin{theorem}\label{Theorem 1/2}
For each $p\in \mathcal P$, there is a self-similar $p$-adic fractal string $\mathcal{L}_p$ of Minkowski--Bouligand dimension $D=\frac{1}{2}$
and with oscillatory period $\p=\frac{\pi}{\log p}$.
\end{theorem}

\begin{corollary}
For each positive integer $v$ and a prime number $p\in \mathcal P$, 
there is a  $p$-adic fractal string $\mathcal{L}_p$ of 
Minkowski--Bouligand dimension $D=\frac{1}{2}\log_p v$.
\end{corollary}

\begin{proof}
For $S$ of cardinality $v$, 
the geometric zeta function of the corresponding $p$-adic fractal string $\mathcal{L}_p$ is given by
\[
\zeta_{\mathcal{L}_p}(s)=
\sum_{n=1}^{\infty}\frac{(p^2-v){v}^{n-1}}{p^{2ns}}
=\frac{p^2-v}{p^{2s}-v},~\mbox{for} ~\Re(s)>\frac{\log v}{\log p^2}.
\]
Therefore, $\zeta_{\mathcal{L}_p}(s)$ 
admits a meromorphic continuation to all of $\mathbb C$, 
given by the same expression as above;
 hence, 
\[
\mathcal{D}_{\mathcal{L}_p}=\Big\{\frac{1}{2}\log_p v+ \frac{\pi i n}{\log p}~\Big |~ n\in \mathbb Z\Big\}
\]
is the set of
complex dimensions of $\mathcal{L}_p$,
as desired.
This completes the proof of the corollary.
\end{proof}

The above construction of a $p$-adic fractal string 
of Minkowski--Bouligand  dimension $D=\frac{1}{2}$
reminds one of the intersection of the graph of Frobenius with the diagonal 
in Bombieri's proof of the Riemann hypothesis for zeta functions of curves over finite fields \cite{MvF}.
This may eventually give rise to a fractal approach to translating his proof of the Riemann hypothesis for zeta functions of curves over finite fields to the curve spec $\mathbb Z$
 over the `field of characteristic one' $\mathbb{F}_1$,
which is the case of the famous Riemann hypothesis for the Riemann zeta function.

The Riemann hypothesis for the Riemann zeta function
\[\zeta(s)=
\prod_{p\in \mathcal{P}}\frac{1}{1-\frac{1}{p^s}}=
\sum_{n\in \mathbb N}\frac{1}{n^s}
\]
states that 
all complex zeros of the Riemann zeta function $\zeta(s)$ lie on the critical line 
\[
\big\{s\in \mathbb C~ |~ \Re(s)=\frac{1}{2}\big\}. 
\] 

We can now generalize Theorem \ref{Theorem 1/2}
in order to obtain the following theorem, which is the main result of this paper. 
 
\begin{theorem}\label{Main Theorem}
For each  $p\in \mathcal P$, there is a self-similar $p$-adic fractal string $\mathcal{L}_p$ of rational dimension $D=\frac{k}{m}\in [0, 1]$
and with oscillatory period $\p=\frac{2\pi}{m\log p}.$ 
\end{theorem}

\begin{proof}
For nonnegative integers $k\leq m,$
an analogous construction to the one given in Theorem~\ref{Theorem 1/2},
but now with a subset $S$ of all $m$-tuples of digits with cardinality~$p^k,$ 
creates a $p$-adic fractal string $\mathcal{L}_p$ of dimension 
$D=\frac{k}{m}$  and with oscillatory period $\p=\frac{2\pi}{m\log p}.$ 
More specifically,
 $\mathcal{L}_p$ is a self-similar $p$-adic fractal string in~$\mathbb{Z}_p$ 
that is fixed by the iterated function system of similarity contraction mappings from $\mathbb{Z}_p$ into itself:
\[
f_{a_M}(x)= a_0+ \cdots + a_{m-1}p^{m-1}+ p^mx, \mbox{ for each } a_{M}:=a_0+\cdots + a_{m-1}p^{m-1} \notin S.
\]
Thus, the self-similar string $\mathcal{L}_p$ has $(p^m-p^k)p^{nk}$ substrings of length $p^{-(n+1)m},$ for $n=0,1,2, \cdots.$
Therefore,
 its geometric zeta function is given by
\[
\zeta_{\mathcal{L}_p}(s)=\frac{p^m-p^k}{p^{ms}-p^k}=\frac{p^{m(1-D)}-1}{p^{m(s-D)}-1},  
~\mbox{for all}~ s\in \mathbb C. 
\]
Hence,
 the set of complex dimensions of $\mathcal{L}_p$ is given by
$\mathcal{D}_{\mathcal{L}_p}=\big\{D+ in\p ~|~ n \in \mathbb Z\big\}$,
where $D=\frac{k}{m}$ is the Minkowski--Bouligand dimension of $\mathcal{L}_p$ and $\p=\frac{2\pi}{\log p^m}$ is its oscillatory period. 
The residue of $\zeta_{\mathcal{L}_p}(s)$ at the rational dimension $D=\frac{k}{m}$ is 
\[\mbox{res}(\zeta_{\mathcal{L}_p}(s); \frac{k}{m})=\frac{p^{m-k}-1}{\log p^m}=\frac{p^{m(1-D)}-1}{\log p^m},\]
and this is also the residue of $\zeta_{\mathcal{L}_p}(s)$ at every other complex dimension $\omega\in \mathcal{D}_{\mathcal{L}_p}$
due to the periodicity of $\zeta_{\mathcal{L}_p}$.
\end{proof}

\begin{remark}
In the archimedean world,
 a real fractal string is defined as a bounded open subset of the real line. 
Such a set has countably many connected components, and the lengths of the string are recovered as the lengths of the components. 
In the nonarchimedean world, this definition fails. 
The $p$-adic fractal string constructed above consists of infinitely many disjoint scaled copies of $\mathbb{Z}_p$ which cover all of $\mathbb{Z}_p.$
Because of the totally disconnected nature of $\mathbb{Z}_p$, there are no boundary points where different lengths are separated from each other. 
Instead, we simply list the sequence of lengths for a $p$-adic fractal string. 
However,
 see also~\cite{LapLu2, LapLu3} where $p$-adic fractal strings,
viewed as bounded open subsets of $\mathbb{Q}_p$,
 are also defined in terms of the lengths of the $p$-adic convex components. 
\end{remark}

\begin{remark}\label{RM1}
Topologically,
 the ring of $p$-adic integers $\mathbb{Z}_p$ is zero-dimensional because it is totally disconnected, 
but here we view it as being one-dimensional;
so that every $p$-adic fractal string $\mathcal{L}_p$ contained in $\mathbb{Z}_p$ has  dimension at most 1. 
The Minkowski--Bouligand dimension of a fractal set  is not a topological invariant, but instead, 
it gives combinatorial and metric information, 
namely, the number of balls of a given size required to cover the set. 
Likewise, the complex dimensions of a fractal string contain information about how much this number oscillates. 
\end{remark}

\subsection{Ad\`elic fractal string with global complex dimension $1/2$}

According to Theorem~\ref{Theorem 1/2},
for each prime number $p\in \mathcal{P}$, there is a self-similar $p$-adic fractal string $\mathcal{L}_p$ of  dimension $D=\frac{1}{2}$.
Therefore,
 the infinite Cartesian product 
\[
 \prod_{p\in \mathcal{P}}\mathcal{L}_p
\]
is a self-similar ad\`elic fractal string $\mathcal{L}_{\frac{1}{2}}$ in the set of finite ad\`eles $\mathbb{A}_0$
because it is fixed by the iterated function system 
\[
\Phi=\{\phi_p\}_{p\in \mathcal P},
\]
where $\phi_p(x)=f_{a+bp}(x)= a+bp+p^2x, ~\mbox{for each}~ a+bp\notin S$.

Since, for each $p\in \mathcal P$, the geometric zeta function of $\mathcal{L}_p$ is 
\[
\zeta_{\mathcal{L}_p}(s)
=\frac{p-1}{p^{2(s-\frac{1}{2})}-1},
 ~\mbox{for all}~ s\in \mathbb C.
\]
However,
 the formal infinite product of meromorphic functions 
\[
 \zeta_{\mathcal{L}_{\frac{1}{2}}}(s)=
\prod_{p\in \mathcal P} \zeta_{\mathcal{L}_p}(s)=\prod_{p\in \mathcal P} \frac{p-1}{p^{2(s-\frac{1}{2})}-1} 
\]
converges only at $s=1$,
 with value given by 
 $ \zeta_{\mathcal{L}_{\frac{1}{2}}}(1)=1$.
 
We propose the following questions about the ad\`elic fractal string $\mathcal{L}_{\frac{1}{2}}$:
What is a proper definition of the geometric zeta function for $\mathcal{L}_{\frac{1}{2}}$?
What is the dimension of  $\mathcal{L}_{\frac{1}{2}}$?
And does it have complex dimensions?
Moreover, in order to study the connection with the Riemann hypothesis for the Riemann zeta function $\zeta(s)$, we also want to ask: 
   What is the spectrum of the vibrating ad\`elic fractal string $\mathcal{L}_{\frac{1}{2}}$?

\section{$p$-adic Cantor strings}

We describe a simple construction of an infinite family of $p$-adic Cantor strings $\mathcal{CS}_p$
in the nonarchimedean ring of $p$-adic integers $\mathbb{Z}_p$
and simultaneously their archimedean counterparts in the real unit interval $[0,1]$, the base-$p$ Cantor string~$\mathcal{CS}_p^*$.
The Minkowski--Bouligand dimensions of the nonarchimedean and archimedean Cantor strings vary from 0 to 1 as the prime number $p$ becomes larger. 
Directly above and below the dimension $D=D_p$ lie infinitely many complex dimensions, periodically distributed along a discrete vertical line 
\[l_D=\{\omega \in \mathcal{D}_{\mathcal{CS}_p}~|~\Re(\omega)=D\}=\mathcal{D}_{\mathcal{CS}_p}.\]
The distribution of complex dimensions on each vertical line $l_D$ is discrete near the dimension $D=0$, but becomes denser as the dimension tends to $1$. 

\begin{lemma}
For  $p=2$, there is a 2-adic fractal string $\mathcal{CS}_2$ of
 Minkowski--Bouligand dimension $D=0$ and  with oscillatory period $\p=\frac{2\pi}{\log 2}$.
 \end{lemma}

 \begin{proof}
 A simple way to create a 2-adic fractal string in the ring of 2-adic integers is 
to decompose $\mathbb{Z}_2$ into a disjoint union of two scaled copies of $\mathbb Z_2$ itself, 
\[\mathbb{Z}_2=2\mathbb{Z}_2\cup1+2\mathbb{Z}_2\]
and then keeping the subinterval $1+2\mathbb{Z}_2$ as the first substring of the 2-adic fractal string $\mathcal{CS}_2$.
 We iterate the above process with the remaining subinterval $2\mathbb{Z}_2,$
 and then pursue the same construction, ad infinitum, in order  
 to create the $2$-adic Cantor string
 \[\mathcal{CS}_2=1+2\mathbb{Z}_2\cup2+4\mathbb{Z}_2\cup4+8\mathbb{Z}_2\cup8+16\mathbb{Z}_2\cup\cdots. \]
 Hence, the geometric zeta function $\zeta_{\mathcal{CS}_2}$  of $\mathcal{CS}_2$ equals the meromorphic continuation of the following convergent geometric series to the entire complex plane 
 $\mathbb C:$
\[
 \frac{1}{2^s}+\frac{1}{4^s}+\frac{1}{8^s}+\frac{1}{16^s}\ldots=\frac{1}{2^s-1},\ \mbox{ for } \Re(s)>0;
\]
 that is to say,
\[
 \zeta_{\mathcal{CS}_2}(s)=\frac{1}{2^s-1},\ \mbox{ for all }s\in \mathbb C.
\]
Therefore, the set of complex dimensions for $\mathcal{CS}_2$ is obtained by solving the  equation $2^s-1=0$.
 The residue of $\zeta_{\mathcal{CS}_2}(s)$ at the least-fractal dimension $D=0$ is given by
\[
 \mbox{res}(\zeta_{\mathcal{CS}_2}; 0)=\frac{1}{\log 2},
\]
 and this is also the residue of $\zeta_{\mathcal{CS}_2}$ at every other complex dimension $\omega \in \mathcal{D}_{\mathcal{CS}_2}$
 by the periodicity of $\zeta_{\mathcal{CS}_2}$.
 \end{proof}
 
 \begin{remark}
 The 2-adic Cantor string $\mathcal{CS}_2$ is the topological complement of the 2-adic Cantor set $\mathcal{C}_2$ in the ring of 2-adic integers $\mathbb{Z}_2$, see \S\ref{SSC}. 
The 2-adic Cantor set $\mathcal{C}_2$ is the unique nonempty compact subset of $\mathbb{Z}_2$ that is invariant under the transformation of the iterated function system $\Phi=\{\phi\}$: 
 $\mathcal{C}_2=\phi(\mathcal{C}_2)$, 
 where $\phi(x)=2x$ is a similarity contraction mapping of $\mathbb{Z}_2$ into itself. 
 So, technically, $\mathcal{C}_2$ is not self-similar because the iterated function system $\Phi=\{\phi\}$ has only one map \cite{Fal}. 
 \end{remark}
 
 \begin{theorem}
 For each $p>2$, there is a self-similar $p$-adic fractal string $\mathcal{CS}_p$ of 
 Minkowski--Bouligand dimension $D=\frac{\log(\frac{1+p}{2})}{\log p}$
 and with oscillatory period $\p=\frac{2 \pi}{\log p}$.
  \end{theorem}
 
 \begin{proof}
 Reminiscent of Smith's construction of the general Cantor set, 
 we create a $p$-adic fractal string in the ring of $p$-adic integers $\mathbb Z_p$ by decomposing $\mathbb{Z}_p$ into the disjoint union
 \[p\mathbb{Z}_p\cup 1+p\mathbb{Z}_p \cup \ldots \cup (p-1)+p\mathbb{Z}_p\]
 and then keeping the subintervals \[1+p\mathbb{Z}_p, 3+p\mathbb{Z}_p, \ldots, (p-2)+p\mathbb{Z}_p\]
 as the first generation of substrings of the $p$-adic fractal string $\mathcal{CS}_p$. 
 By iterating this process with the remaining subintervals \[p\mathbb{Z}_p, 2+p\mathbb{Z}_p, \ldots, (p-1)+p\mathbb{Z}_p,\]
 we obtain the $p$-adic Cantor string
 \[
 \mathcal{CS}_p= 
\bigcup_{j=1}^{\frac{p-1}{2}}2j-1+p\mathbb{Z}_p\cup \bigcup_{j=1}^{\frac{p-1}{2}}2j-1+p^2\mathbb{Z}_p\cup \cdots. 
 \]
 Thus,
 we get
    $\mu_n= (\frac{p-1}{2})(\frac{p+1}{2})^{n-1}$ subintervals of length $l_n=p^{-n},$ for each positive integer $n$. 
     Hence,
 the geometric zeta function $\zeta_{\mathcal{CS}_p}(s)$ of the $p$-adic Cantor string~$\mathcal{CS}_p$ 
    coincides with the meromorphic continuation to the whole complex plane $\mathbb C$
    of the following convergent geometric series:
    \[
    \sum_{n=1}^{\infty}\mu_n\cdot l_n^s=
    \frac{p-1}{p+1}\sum_{n=1}^{\infty} \Big(\frac{1+p}{2p^s}\Big)^n=\frac{p-1}{2p^s-p-1},
~ \mbox{for}~ \Re(s)>\frac{\log(\frac{1+p}{2})}{\log p};\]
that is to say,
\begin{equation}\label{zeta}
\zeta_{\mathcal{CS}_p}(s)= \frac{p-1}{2p^s-p-1},~\mbox{for all} ~ s\in \mathbb{C}.
\end{equation}
Therefore, the set of complex dimensions of the $p$-adic Cantor string $\mathcal{CS}_p$ is given by
\[
\mathcal{D}_{\mathcal{CS}_p}=\{D+in\p ~|~ n\in \mathbb{Z}\},
\]
where 
$D=\frac{\log(\frac{1+p}{2})}{\log p}$ is the Minkowski--Bouligand dimension of $\mathcal{CS}_p$ and $\p=\frac{2\pi}{\log p}$ is its oscillatory period. 
The residue of  $\zeta_{\mathcal{CS}_p}$ at the fractal dimension $D=\frac{\log(\frac{1+p}{2})}{\log p}$ is given by
\[\mbox{res}(\zeta_{\mathcal{CS}_p}(s); D)=\frac{p-1}{(p+1)\log p}\]
and
this is also the residue of $\zeta_{\mathcal{CS}_p}$ at every other complex dimension $\omega \in \mathcal{D}_{\mathcal{CS}_p}$,
 by the periodicity of $\zeta_{\mathcal{CS}_p}$;
see (\ref{zeta}) above.
 \end{proof}
 
 \begin{remark}
 Since the limit of $D=\frac{\log(\frac{1+p}{2})}{\log p}$ is 1
 and the limit of $\p=\frac{2 \pi}{\log p}$ is 0,
 as~$p\rightarrow \infty,$
 the periodic distribution of complex dimensions on the vertical line $l_D$ is discrete near the dimension $D=0,$ but becomes denser as the dimension $D$ tends to~1,
as~$p\to\infty$.
 \end{remark}
 
 \begin{remark}
  We recover the 3-adic Cantor string in \cite{LapLu1} when we choose~\mbox{$p=3$}.
 When $p=5$,
 we recover the 5-adic Cantor string in \cite{KRC} (building on~\cite{LapLu1}),
which motivated our construction of the $p$-adic Cantor strings as well as of their archimedean counterparts in \S\ref{SSC}.
 \end{remark}
 
 \subsection{Geometric waves}
 Nonreal complex dimensions are the source of oscillations in the geometry of a fractal string. 
 The real parts of the complex dimensions correspond to the amplitudes of geometric waves propagating through the `space of scales' 
 that lies beneath the  surface of a fractal string
 and the imaginary parts of the complex dimensions correspond to the frequencies of those geometric waves. 
 The lengths out of which a fractal string is composed of can be thought of as being the underlying scales of the system; see \cite{Lap5}.
 
For each complex dimension $\omega=D+in\p\in \mathcal{D}_{\mathcal{CS}_p}$, 
the residue of  $\zeta_{\mathcal{CS}_p}$ at $\omega$ is independent of $n\in \mathbb Z$ and is given by
$\mbox{res}(\zeta_{\mathcal{CS}_p}; \omega)=\frac{p-1}{(p+1)\log p}$.
Therefore, for any sufficiently small $\varepsilon>0$, 
the volume of the inner $\varepsilon$-neighborhoods of the $p$-adic Cantor string $\mathcal{CS}_p$ is given by the explicit fractal tube formula
\[
V_{\mathcal{CS}_p}(\varepsilon)=
\frac{p-1}{(p+p^2)\log p}\times \varepsilon^{1-D}
\sum_{n\in\mathbb Z}
\frac{\cos(n\p \log \varepsilon)-i\sin(n\p \log \varepsilon)}{1-D-in\textbf{p}},
\]
which follows from the exact fractal tube formula for any self-similar $p$-adic fractal string $\mathcal{L}_p$ with simple complex dimensions (see \cite{LLvF1, LLvF2}):
\[
V_{\mathcal{L}_p}(\varepsilon)=
\sum_{\omega \in \mathcal{D}_{\mathcal{L}_p}}
\frac{\mbox{res}(\zeta_{\mathcal{L}_p}; ~\omega)}{p}\times \frac{\varepsilon^{1-\omega}}{1-\omega}.
\]
Consequently, every $p$-adic Cantor string has logarithmic oscillations of order $D$ in its geometry because the limit of
$\frac{V_{\mathcal{CS}_p} (\varepsilon)}{\varepsilon^{1-D}}$
does not exist in $(0, +\infty),$ as $\varepsilon \rightarrow 0^+$.
Therefore, $p$-adic Cantor strings are not Minkowski measurable and hence, their Minkowski contents do not exist.
However,
the average Minkowski content,
 defined as a Ces\`aro logarithmic average,
 exists for each $p$-adic Cantor string $\mathcal{CS}_p$
and is given by
\[
\mathcal{M}_{av}(\mathcal{CS}_p)=
\frac{p-1}{(p+p^2)\log(\frac{2p}{p+1})},
\]
while $\mathcal{M}_{av}(\mathcal{L}_p)$ is defined by the formula
\[
\mathcal{M}_{av}(\mathcal{L}_p):=\lim_{T\rightarrow \infty}\frac{1}{\log T}\int_{T^{-1}}^1\frac{V_{\mathcal L_p}(\varepsilon)}{\varepsilon^{1-D}} 
\frac{d\varepsilon}{\varepsilon}=
\frac{1}{p(1-D)}\mbox{res}(\zeta_{\mathcal{L}_p}; D),
\]
which is valid for any self-similar $p$-adic fractal string $\mathcal{L}_p$ (see \cite{LLvF1}).

\subsubsection{$V_{\mathcal{L}_p}(\varepsilon)$ associated with a geometric wave}

For each prime  $p\in \mathcal P$ and a sufficiently small $\varepsilon >0$, 
the volume $V_{\mathcal{L}_p}(\varepsilon)$ of the inner $\varepsilon$-neighborhoods of the $p$-adic fractal string $\mathcal{L}_p$ of rational dimension $D=\frac{k}{m}$ and with oscillatory period $\p=\frac{2\pi}{m\log p}$
 is given by the following exact
 explicit fractal tube formula (see \cite{LLvF1, LLvF2}),
\[
V_{\mathcal{L}_p}(\varepsilon)=
\frac{p^{m(1-D)}-1}{p\log {p^m}}\times \varepsilon^{1-D}
\sum_{n\in\mathbb Z}
\frac{\cos(n\p \log \varepsilon)-i\sin(n\p \log \varepsilon)}{1-D-in\textbf{p}},
\]
since the residue of $\zeta_{\mathcal{L}_p}$ 
at every complex dimension $\omega=D+in\p \in \mathcal{D}_{\mathcal{L}_p}$ is given by 
\[\mbox{res}(\zeta_{\mathcal{L}_p}(s); \omega)=\frac{p^{m(1-D)}-1}{\log {p^m}}.\]

Consequently,
the $p$-adic fractal string $\mathcal{L}_p$ is not Minkowski measurable.
However,
its average Minkowski content exists and is given by
\[
\mathcal{M}_{av}(\mathcal{L}_p)=
\frac{p^{m(1-D)}-1}{p \log p^{m(1-D)}}.
\]

The notion of Minkowski content is a generalization of the notion of volume of a smooth manifold in $N$-dimensional Euclidean space to an arbitrary 
 bounded subset of $\mathbb{R}^N$
of any fractal dimension $D\in [0, N]$.
Intuitively, the $D$-dimensional Minkowski content of a fractal set of Minkowski--Bouligand dimension~$D$ is its $D$-dimensional fractal volume.
Motivated in parts by some of the main results of \cite{LapPo1} and \cite{Lap3.5},
Alain Connes showed
that the Minkowski content is a natural analogue of the volume of a compact smooth Riemannian spin manifold for a fractal set
in the case of the ordinary Cantor set or string \cite{Connes}. 
Benoit Mandelbrot proposed the Minkowski content as a measure of fractal lacunarity  
because the value of the Minkowski content allows one to compare the lacunarity of fractal sets of the same Minkowski--Bouligand dimension \cite{Mandelbrot}.
(See also \cite{KK,LapPo1,LvF1,LvF2,LvF3,LRZ}.)

On the other hand,
 Minkowski measurability is a kind of `fractal regularity' for the underlying geometry. 
The existence of nonreal complex dimensions of a fractal string $\mathcal L$, 
along with the simplicity of the  dimension $D$ as pole of its geometric zeta function $\zeta_{\mathcal{L}}(s)$,
 determines the Minkowski nonmeasurability of  $\mathcal L$ \cite{LvF1, LvF2, LvF3};
 this result has been extended to higher-dimensional fractals in \cite{LRZ}.

\subsection{$p$-adic Cantor sets and their archimedean counterparts}\label{SSC}

Fix a prime~$p$.
The $p$-adic Cantor string $\mathcal{CS}_p$ is the set-theoretic complement of the self-similar $p$-adic Cantor set $\mathcal{C}_p$ in the ring of $p$-adic integers $\mathbb{Z}_p$. 
The archimedean counterpart in the archimedean field of real numbers $\mathbb R$ of the nonarchimedean $p$-adic Cantor set $\mathcal{C}_p$ is the base-$p$ Cantor set $\mathcal{C}_p^*$.
The complement of the base-$p$ Cantor set $\mathcal{C}_p^*$  in the real unit interval $[0, 1]$ is the base-$p$ Cantor string $\mathcal{CS}_p^*$.

\begin{lemma}
 For each prime\/ $p>2,$
the\/ $p$-adic Cantor set\/ $\mathcal{C}_p$ is a self-similar set in\/ $\mathbb{Z}_p;$
that is to say,\/
 $\mathcal{C}_p$ is the unique nonempty invariant compact set of an iterated function system of\/ $(p+1)/2$ similarity contraction mappings from\/ $\mathbb{Z}_p$ into itself\/{\rm:}
\[
\mathcal{C}_p=\Phi_p(\mathcal{C}_p)=\bigcup_{i=1}^{\frac{p+1}{2}}\phi_i(\mathcal{C}_p),
\]
where  
$\Phi_p=\{\phi_1, \phi_2, \ldots, \phi_{\frac{p+1}{2}}\}= \{0+px, 2+px, \ldots, p-1+px\}.$ 
\end{lemma}

\begin{proof}
 This follows from Part (i) of Theorem 4.2 in \cite{LapLu2}.
\end{proof}
 
\begin{lemma}
For each\/ $p>2,$
 the $p$-adic Cantor set\/ $\mathcal{C}_p$ is a subset of the\/ $p$-adic integers\/ $\mathbb{Z}_p$ whose elements
only contain even digits in their\/ $p$-adic expansions\/{\rm:}
\[
\mathcal{C}_p=\{a_0 + a_1p+a_2 p^2 + a_3p^3+\cdots \in \mathbb{Z}_p ~|~ a_i \in \{0, 2, 4, \ldots, p-1 \},\mbox{ \rm for all } i\geq 0 \}.
\]
\end{lemma}

\begin{proof}
 Since we remove all the odd digits at every stage in the construction of the $p$-adic Cantor set $\mathcal{C}_p$,
 the elements of $\mathcal{C}_p$ must consist only of even digits in their $p$-adic expansions;
see \cite{LapLu1, KRC}.
\end{proof}
 
 Let $\mathcal{C}_p^*$ be a subset of the unit interval $[0,1]$ on the real line whose elements only contain even digits in their base-$p$ expansion:
  \[ \mathcal{C}_p^*:=\{a_0 + a_1p^{-1}+a_2 p^{-2} + a_3p^{-3}+\cdots  ~|~ a_i \in \{0, 2, 4, \ldots, p-1 \}, ~\mbox{for all} ~ i\geq 0 \}.\]
 Then, the base-$p$ Cantor set $\mathcal{C}_p^*$ is homeomorphic to the $p$-adic Cantor set $\mathcal{C}_p$ via the continuous map
 \[
 \sum_{i=0}^{\infty}a_i\cdot p^{-i} \longmapsto  \sum_{i=0}^{\infty}a_i\cdot p^i
 \]
 that sends each element of the compact set 
  $\mathcal{C}_p^*$ in the complete metric space $(\mathbb R, |\cdot|_{\infty})$  to an element of the compact set $\mathcal{C}_p$ in the complete ultrametric space $(\mathbb{Q}_p, |\cdot|_p)$.
 Therefore, the  base-$p$ Cantor set $\mathcal{C}_p^*$ can be considered as a natural archimedean counterpart 
  of the nonarchimedean $p$-adic Cantor set $\mathcal{C}_p$.
    
 The complement of the base-$p$ Cantor set $\mathcal{C}_p^*$ in the real unit interval $[0,1]$ is an archimedean fractal string  $\mathcal{CS}_p^*$, called the base-$p$ Cantor string.
 Hence, the base-$p$ Cantor string $\mathcal{CS}_p^*$ is the subset of the unit interval $[0,1]$ whose elements must contain at least one odd digit in their base-$p$ expansions. 
In the same way,
 the complement of the $p$-adic Cantor set $\mathcal{C}_p$ in the ring of $p$-adic integers $\mathbb{Z}_p$ is the $p$-adic Cantor string $\mathcal{CS}_p$.
Therefore,
 every element of the $p$-adic Cantor string $\mathcal{CS}_p$ contains at least one odd digit in its $p$-adic expansion. 

The base-$p$ Cantor string $\mathcal{CS}_p^*$ and  the $p$-adic Cantor string $\mathcal{CS}_p$ both have the same sequence of lengths with respect to each respective metric.
Hence,
the base-$p$ Cantor string $\mathcal{CS}_p^*$ can be considered as a natural archimedean counterpart of the nonarchimedean $p$-adic Cantor string $\mathcal{CS}_p$.
Therefore,
 we consider the cartesian product $\mathcal{CS}_p^*\times\mathcal{CS}_p$ as a natural 
symmetric fractal string in the archimedean--nonarchimedean space $\mathbb R\times \mathbb{Q}_p$.

\subsection{Ad\`elic Cantor string $\mathcal{CS}_{\mathbb{A}_0}$}

For each $p\in \mathcal P$, let $\mathcal{CS}_p$ be the $p$-adic Cantor string. 
Then, an infinite cartesian product of every $p$-adic Cantor string is an ad\`elic fractal string $\mathcal{CS}_{\mathbb{A}_0}$
in the set of finite ad\`eles $\mathbb{A}_0$:
\[
\mathcal{CS}_{\mathbb {A}_0}=\prod_{p\in \mathcal P}\mathcal{CS}_p.
\]

\begin{theorem}
The ad\`elic Cantor string\/ $\mathcal{CS}_{\mathbb{A}_0}$ is a self-similar string in\/ $\mathbb{A}_0$.
\end{theorem}

\begin{proof}
Let $\Phi=\{\Phi_2, \Phi_3, \Phi_5,  \ldots\}$ be an iterated function system in the set of finite ad\`eles $\mathbb{A}_0$, where
each $\Phi_p=\{px, 2+px, \ldots, p-1+px\}$ is an iterated function system of similarity contraction mappings in $\mathbb{Z}_p$ 
that acts trivially on every other  components of $\mathbb{A}_0$.
Moreover, let 
$\mathcal{C}_{\mathbb{A}_0}=\mathcal{C}_2 \times \mathcal{C}_3\times \mathcal{C}_5 \times \dots$ 
be the ad\`elic Cantor set in $\mathbb{A}_0$.
Then,
it is easy to check that $\Phi$ is a contraction
with respect to the Hausdorff metric induced by the standard norm 
$\prod_{p=2}^{\infty}|\cdot|_p$ on $\mathbb{A}_0$,
and
$\Phi(\mathcal{C}_{\mathbb{A}_0})=\mathcal{C}_{\mathbb{A}_0}$
(which is the natural definition of a self-similar fractal)
since
\[
\Phi(\mathcal{C}_{\mathbb{A}_0})
= \Phi_2(\mathcal{C}_2 )\times \Phi_3(\mathcal{C}_3) \times  \Phi_5(\mathcal{C}_5) \times \ldots
= \mathcal{C}_2 \times \mathcal{C}_3  \times \mathcal{C}_5 \times \dots
=\mathcal{C}_{\mathbb{A}_0}.
\]
Thus, the ad\`elic Cantor set $\mathcal{C}_{\mathbb{A}_0}$ is self-similar. 
\end{proof}

\subsection{Ad\`elic Cantor--Smith string $\mathcal{CS}_{\mathbb A}$}

For each positive integer $m>2$,
let~$\mathcal{C}_m$ be the general Cantor set as constructed by Smith in 1875 \cite{Smith}.
Then,
the topological complement of $\mathcal{C}_m$ in $[0,1]$ is an ordinary fractal string $\mathcal{CS}_m$ consisting of all the deleted segments in the construction of $\mathcal{C}_m$. 
Thus,
 $\mathcal{CS}_m$ can be represented as a sequence of lengths $l_n=m^{-n}$ with multiplicity $\mu_n=(m-1)^{n-1}$,
 for each positive integer $n$.
Therefore,
the geometric zeta function $\zeta_{\mathcal{CS}_m}$ for $\mathcal{CS}_m$ is given by 
\[
\zeta_{\mathcal{CS}_m}(s)=\frac{1}{m^s+1-m},\ \mbox{ for all }s\in \mathbb C. 
\]
When $m=3$, $\mathcal{CS}_3$ is the ordinary Cantor string $\mathcal{CS}$ as constructed and studied by the first author and C.\ Pomerance in \cite{Lap2, Lap3,LapPo1} and then expounded upon in \cite{LvF1, LvF2, LvF3}.

The Cartesian product of $\mathcal{CS}_m$ 
and the ad\`elic Cantor string $\mathcal{CS}_{\mathbb{A}_0}$ 
is a self-similar ad\`elic string $\mathcal{CS}_{\mathbb A}$
in the ring of ad\`eles $\mathbb{A}$. 
We call it
 the ad\`elic Cantor--Smith string.  

The formal infinite product of meromorphic functions 
\[
\zeta_{\mathcal{CS}_{\mathbb A}}(s)=
\zeta_{\mathcal{CS}_m}(s)\cdot
\zeta_{\mathcal{CS}_{\mathbb{A}_0}}(s)=
\frac{1}{m^s+1-m}\cdot  \frac{1}{2^s-1}
\prod_{p>2}\frac{p-1}{2p^s-p-1} 
\]
could be a geometric zeta function $\zeta_{\mathcal{CS}_{\mathbb A}}$ for the ad\`elic Cantor--Smith string $\mathcal{CS}_{\mathbb A}$.
However, this product converges only for $s=1$,
 with 
 $\zeta_{\mathcal{CS}_{\mathbb A}}(1)=1$.

\section{Ad\`elic Euler string $\mathcal{E}_{\mathbb{A}_0}$}

For each  $p\in\mathcal{P}$, 
let $\mathcal{E}_p=\bigcup_{n=0}^{\infty}(a_n+p^n\mathbb{Z}_p)$ 
be the $p$-adic Euler string 
with the geometric zeta function
\[
\zeta_{\mathcal{E}_p}(s)=\frac{1}{1-\frac{1}{p^s}}
\]
as constructed by the authors in \cite{LLvF2}.
Then,
the infinite Cartesian product
\[
\prod_{p\in\mathcal{P}}\mathcal{E}_p
\]
can be considered as an ad\`elic fractal string $\mathcal{E}_{\mathbb{A}_0}$ in the set of finite ad\`eles $\mathbb{A}_0$. 
We call $\mathcal{E}_{\mathbb{A}_0}$  the ad\`elic Euler string. 
The ad\`elic Euler string $\mathcal{E}_{\mathbb{A}_0}=\prod_{p}\mathcal{E}_p$ is not self-similar since every $\mathcal{E}_p$ is not self-similar because its dimension is zero. 

Let $\zeta_{\mathcal{E}_p}(s)$ be the geometric zeta function of the $p$-adic Euler string $\mathcal{E}_p$, 
then the infinite product of complex meromorphic functions
\[
\prod_{p\in\mathcal{P}} \zeta_{\mathcal{E}_p}(s)=\prod_{p\in\mathcal{P}}\frac{1}{1-\frac{1}{p^s}}=\sum_{n\in \mathbb N}\frac{1}{n^s}
\] 
is the Riemann zeta function $\zeta(s)$. 
Therefore, we may consider the Riemann zeta function $\zeta(s)$ 
as the geometric zeta function $\zeta_{\mathcal{E}_{\mathbb{A}_0}}(s)$ of the ad\`elic Euler string $\mathcal{E}_{\mathbb{A}_0}$; 
that is to say:
\[
\zeta_{\mathcal{E}_{\mathbb{A}_0}}(s)=\zeta(s),\ \mbox{ for all }s\in \mathbb C. 
\]
With this in mind,
 we may consider the only simple pole of the Riemann zeta function $\zeta(s)$ at $s=1$ as the global Minkowski-Bouligand dimension of the ad\`elic Euler string 
$\mathcal{E}_{\mathbb{A}_0}$.

\subsection{Ad\`elic Euler--Riemann string $\mathcal{E}_{\mathbb A}$}

Let $\mathcal{E}_{\mathbb{A}_0}=\prod_{p}\mathcal{E}_p$
 be the ad\`elic Euler string and the local positive measure
$h=\sum_{n\in \mathbb N}\delta_{\{n\}}$
 be the harmonic string (as in \cite{LvF1, LvF2, LvF3}),
then the Cartesian product 
\[
\mathcal{E}_{\mathbb A}=
h\times \mathcal{E}_{\mathbb{A}_0}
\]
can be considered as an ad\`elic fractal string  in the ring of ad\`eles $\mathbb{A}$. 
We call it the ad\`elic Euler--Riemann string.

Let $\zeta_{\mathcal{E}_{\mathbb{A}_0}}(s)$ be the geometric zeta function of the ad\`elic Euler string $\mathcal{E}_{\mathbb {A}_0}$ 
 and $\zeta_{h}(s)$ be the geometric zeta function of the harmonic string $h$, 
then their product 
\[
\zeta_{h}(s)\cdot \zeta_{\mathcal{E}_{\mathbb{A}_0}}(s)=\zeta(s)\cdot\zeta(s) 
\]
is the square of the Riemann zeta function,
which could be viewed as the geometric zeta function $\zeta_{\mathcal{E}_{\mathbb{A}}}(s)$
for the ad\`elic Euler--Riemann string $\mathcal{E}_{\mathbb{A}}$.

\section{Epilogue}

In \emph{Number Theory as the Ultimate Physical Theory}, Igor  Volovich has suggested that $p$-adic numbers can possibly be used to describe the geometry of spacetime at very high energies, and hence, very small scales, 
because measurements in the `archimedean' geometry of spacetime at fine scales have no certainty \cite{V}. 
Moreover, Stephen Hawking and other authors have also suggested that the fine scale structure of spacetime may be fractal \cite{GH, HI, N, WF}. 
Furthermore, 
the first author has suggested that 
fractal strings and fractal membranes may be related to T-duality in string theory 
and the functional equation of the Riemann zeta function~$\zeta(s)$; 
and that the $p$-adic and ad\`elic analogues of these notions may be helpful for understanding the underlying noncommutative spacetimes and their moduli spaces~\cite{Lap4}. 

The global theory of complex dimensions for ad\`elic fractal strings would
provide a natural and unified framework for understanding the vibrations and resonances in the geometry and the spectrum of ad\`elic fractal strings 
as well as  the pole and zeros of the Riemann zeta function $\zeta(s)$. 
The theory would also precisely describe the oscillatory nature intrinsic to the geometry of ad\`elic fractal strings 
as geometric waves propagating through the space of scales that lies beneath the surface of the ad\`elic fractal strings.
Moreover,
 it would  shed light on the Riemann hypothesis for the Riemann zeta function $\zeta(s)$ via the inverse spectral problems for ad\`elic fractal strings 
\cite{LapMa1,Lap3,HerLap, Lap4, LvF1, LvF2, LvF3, MvF, MvF3}.

\bibliographystyle{amsplain}

\end{document}